\documentclass[letterpaper, reqno, 11pt]{amsart}
\usepackage[english]{babel}
\usepackage{graphicx,color}
\usepackage[all]{xy}
\usepackage{amsfonts,dsfont,mathrsfs}
\usepackage{enumerate}
\usepackage{amsmath,amscd}
\usepackage{amsmath}
\usepackage{amssymb}
\usepackage{amsthm}
\usepackage{tikz} 
\usepackage{tikz-cd}
\usepackage{graphics}
\usepackage{listings}
\usepackage{xcolor}
\usepackage{gensymb}
\usepackage{amssymb}
\usepackage{setspace}
\usepackage{fancyhdr}
\usepackage{dsfont}
\usepackage{bm}
\usepackage{enumitem}

\numberwithin{equation}{section}

\newtheorem{thm}{Theorem}

\newtheorem{prop}{Proposition}

\newtheorem{lem}{Lemma}

\newtheorem{cor}{Corollary}

\theoremstyle{definition}

\newtheorem*{notation}{Notations}

\thanks{}

\begin{document}

\title{Estimates for the norm of the derivative of Lie exponetial map for connected Lie groups}
\author[R. Bidar]{Reza Bidar}
\address{Department of Mathematics\\
                 Albion College\\
                 Albion, MI 49224}
\email{mbidar@albion.edu}                 
\date{\today}

\begin{abstract}

Let $G$ be a real connected Lie group with a left invariant metric $d$, $\mathfrak{g}$ its Lie algebra. In this paper we present a set of interesting upper and lower bounds for $|d\exp_{x}(y)|,\  x,y \in \mathfrak{g}$. If $\textrm{ad}_x$ is diagonalizable, these bounds only depend on eigenvalues of $\textrm{ad}_x$, but in general they are functions of the singular values $\textrm{ad}_x$.


\end{abstract}

\maketitle


\section{Introduction}

Let $G$ be a connected real Lie group with a left invariant Riemannian metric $d$, and $\mathfrak{g}$ be its Lie Algebra as an inner product space, and $\exp: \mathfrak{g} \rightarrow G$ the Lie exponential map. For $g \in G$ let $l_g$ denote the left multiplication by $g$. One important question that arise about the exponential map would be asking if there are conditions under which the exponential map is quasi-isometry. This is trivially true if the universal covering of $G$ is $\mathds{R}^n$.  The other conditions that might be worthy of investigation are when $G$ is compact, semi-simple, solvable or nilpotent. In this paper we present general lower and upper bounds for the norm of the differential of the exponential map which provides valuable information regarding the behavior of the exponential map. These bounds show that in general, the exponential map for these types of Lie groups is not a quasi-isometry. \\ 
\par
Given a non-zero vector $x \in \mathfrak{g}$, it is well known that the differential of the exponential map at $x$ is given by
\begin{equation}
    d\exp_{x}=dl_{\exp(x)} \frac{1-e^{-\textrm{ad}_x}}{\textrm{ad}_x}
\end{equation}
Rossman \cite[p.15]{Lie Groups-An Intro}. Since the metric $d$ is left invariant, it follows that for any vector $y \in \mathfrak{g}$ 
\begin{equation*}
    \left|d\exp_{x}(y)\right|=\left| \frac{1-e^{-\textrm{ad}_x}}{\textrm{ad}_x}(y)\right|\, .
\end{equation*} 
Thus the problem of finding upper and lower bounds for $\left|d\exp_{x}(y)\right|$ would be equivalent to finding estimates for the norm of the image of 
\begin{equation*}
    \frac{1-e^{-\textrm{ad}_x}}{\textrm{ad}_x} \ ,
\end{equation*}
which can be regarded as a compact operator on $\mathfrak{g}$ as a finite dimensional Hilbert space. 
When $\textrm{ad}_x$ is diagonalizable it is possible to bound the differential of exponential map by the biggest and the smallest eigenvalues of $(1-e^{-\textrm{ad}_x})/\textrm{ad}_x$ as stated in the following theorem.

\begin{thm}\label{thm-main}
Let $x \in \mathfrak{g}$ be non-zero and such that $\emph{ad}_x$ is diagonalizable. Let $\hat{x}=x/|x|$, $\lambda_1,\cdots,\lambda_{p} \in \mathds{C}$ be non-zero eigenvalues of $\emph{ad}_{\hat{x}}$, and

\begin{equation*}
    \tilde{\lambda}_{\min}(|x|)=\min \left\{1,\left| \frac{1-e^{-\lambda_{1}|x|}}{\lambda_{1}|x|}\right|,\cdots, \left|\frac{1-e^{-\lambda_{p}|x|}}{\lambda_{p}|x|}\right|\right\},
\end{equation*}

\begin{equation*}
\tilde{\lambda}_{\max}(|x|)=\max \left\{1,\left| \frac{1-e^{-\lambda_{1}|x|}}{\lambda_{1}|x|}\right|,\cdots, \left|\frac{1-e^{-\lambda_{p}|x|}}{\lambda_{p}|x|}\right|\right\}\,.
\end{equation*}
Then there exist positive constants $C,D$, only depending on $\hat{x}$, such that for any unit vector $y \in \mathfrak{g}$
\begin{equation}
   C\tilde{\lambda}_{\min}(|x|) \leq \left| d\exp_{x} (y) \right| \leq D \tilde{\lambda}_{\max}(|x|)\, .
\end{equation}
\end{thm} 
For a fix  unit vector $\hat{x}$, if $d\exp_{\hat{x}}$ is invertible, it's easy to find a constant $C_0$ such that for all $t \geq 1$:
    \begin{equation*}
        \left|1-e^{-\lambda_{j}t}\right| \geq C_0,\  1 \leq j \leq p \, .
    \end{equation*}
 This leads us to the following corollary: 
\begin{cor}
Assume $\hat{x}$ is a unit vector, $\emph{ad}_{\hat{x}}$ diagonalizable, $d\exp_{\emph{ad}_{\hat{x}}}$ invertible. Then there exist a positive constant $C$ such that for all $t \geq 1$ and any unit vector $y \in \mathfrak{g}$
\begin{equation}
     \left| d\exp_{t\hat{x}} (y) \right| \geq \frac{C}{t}\, .
\end{equation}
\end{cor}
If $\textrm{ad}_x$ is not diagonalizable, the bounds in above theorem will not work in gerenal. However, by the minimax theorem for singular values (see Proposition \ref{mini-max}), maximum and minimum of $|d\exp_x(y)|$ taken over all unit vectors $y$ are indeed the smallest and the largest singular values of $(1-e^{-\textrm{ad}_x})/\textrm{ad}_x$. This would lead us to the bounds stated in the following theorem.
\begin{thm}\label{thm-MAIN}
Let $x \in \mathfrak{g}$ be non-zero, $\hat{x}=x/|x|$, and $\lambda_1,\cdots,\lambda_{p}$ be non-zero eigenvalues of $\emph{ad}_{\hat{x}}$, $\tilde{\lambda}_{\min}(|x|), \tilde{\lambda}_{\max}(|x|)$ be defined as in Theorem 1, $\mathfrak{g}_{\mathds{C}}=\mathfrak{g}\otimes \mathds{C}$ be the complexification of $\mathfrak{g}$. \\
Let $\displaystyle \delta_0=\max \left\{ \left|\emph{ad}_{x_1}(y_1)\right|:\ x_1,y_1 \in\mathfrak{g}_{\mathds{C}},\  |x_1|,|y_1|=1 \right\}$.\\ We have the following bounds:
\begin{enumerate}
    \item For all unit vectors $y$, 
    \begin{equation}
      \left( \frac{e^{\delta_0|x|}-1}{\delta_0|x|} \right)^{1-n}\prod_{j=1}^p\left| \frac{1-e^{-\lambda_{j}|x|}}{\lambda_{j}|x|}\right|  \leq \left| d\exp_{x} (y) \right| \leq \frac{e^{\delta_0|x|}-1}{\delta_0|x|}
    \end{equation}
    
    \item For every $x$, there are unit vectors $y_0,y_1 \in \mathfrak{g}$ such that 
    \begin{equation}
    \left| d\exp_{x} (y_0) \right| \leq \tilde{\lambda}_{\min}(|x|),\  \tilde{\lambda}_{\max}(|x|) \leq \left| d\exp_{x} (y_1) \right| 
    \end{equation}
\end{enumerate}
\end{thm}

When $\textrm{ad}_x$ is nilpotent, $(1-e^{-\textrm{ad}_x})/\textrm{ad}_x$ would a polynomial in terms of $|x|$ which leads us to the bounds stated in the following proposition.

\begin{prop}\label{exp-nilpotent}
Let $\mathfrak{g}$ be $p$-step nilpotent. Given a non-zero $x \in \mathfrak{g}$,  $|d\exp_x|=O(x^{p-1})$. There exist a polynomial $Q$ with $\deg(Q)\leq (n-1)(p-1)$, such that for all non-zero $x \in \mathfrak{g}$,  
\begin{equation}
    \min_{|y|=1} |d\exp_x(y)| \geq \frac{1}{Q(|x|)}\, .
\end{equation}
\end{prop}

\section{Preliminaries}
To derive estimates on the norm of the differential of the exponential map, we will need some facts from functional calculus. Let $\mathcal{H}$ be a complex Hilbert space. Singular values of a linear bounded operator $T$ defined on Hilbert space $\mathcal{H}$ are the square roots of non-negative eigenvalues of the self-adjoint operator $T^{*}T$. Given a linear operator $T$, we enumerate the singular values $\{s_j(T) \},\ j=1,2,\cdots$ in a non-increasing order, and the eigenvalues $\{\lambda_j(T) \},\ j=1,2,\cdots$ so that the moduli are non-increasing. For a normal operator $T$, singular values are the absolute of eigenvalues. 

\par 
Assume that $\mathcal{H}$ is finite dimensional, $\dim \mathcal{H}=n$. We will use the following facts:
\begin{itemize}
    \item As a direct consequence of the \textit{Spectral Mapping Theorem} (Rudin \cite[Theorem 10.33]{Functional-Rudin}) for every complex function $f$ which is holomorphic on a domain including all eigenvalues of $T$: 
\begin{equation}\label{SpectralMT}
   \lambda_j(f(T))=f(\lambda_j(T)),\  1 \leq j \leq n\, . 
\end{equation}
   \item If $T$ is diagonal with respect to an orthonormal basis, then $T$ is normal and thus $s_j(T)=|\lambda_j(T)|, \ 1 \leq j \leq n$.
\end{itemize}

\par
Singular values of a linear operator defined on a finite dimensional Hilbert space are related to the norm of the operator on some subspaces of $\mathcal{H}$. This is known and minimax principle for singular values, Bhatia \cite[p.75]{Matrix-Analysis}:

\begin{prop}[The minimax principle for singular vlaues]\label{mini-max}
Given any operator on a finite dimensional Hilbert space $\mathcal{H}$, $\dim \mathcal{H}=n$, 
\begin{equation}
    \begin{split}
          s_j(T)& = \max_{\mathcal{M}:\dim \mathcal{M}=j} \min_{x \in \mathcal{M}, |x|=1} |T(x)| \\
         & =\min_{\mathcal{N}:\dim \mathcal{N}=n-j+1} \max_{x \in \mathcal{N}, |x|=1} |T(x)|
    \end{split}
    \end{equation}
for $1 \leq j \leq n$. 
\end{prop}
In particular, the minimax principle implies that:
\begin{equation}
    \max_{|x|=1} |T(x)|=s_1(T),\quad \min_{|x|=1} |T(x)|=s_n(T)\, .
\end{equation}

We also need the following proposition, often known as Weyl's inequality, Birman \cite[p.258]{Spectral-Self-Adjoint}: 

\begin{prop}
Let $\{\lambda_k(T) \}$ be the sequence of eigen-values of a compact operator $T$ on a Hilbert space $\mathcal{H}$, enumerated so that the moduli are non-increasing, and $\{s_k(T) \}$ be its singular values in a non-increasing order. Then 
\begin{equation}\label{Weyl}
    \prod_{1}^r |\lambda_k(T)| \leq \prod_{1}^r s_k(T), \quad r=1,2,\cdots
\end{equation}
\end{prop}
Letting $r=1$, we conclude
\begin{equation}\label{eigen<sing}
    |\lambda_1(T)| \leq |s_1(T)|\, .
\end{equation}
In addittion, if $\mathcal{H}$ is finite dimensional, $\dim \mathcal{H}=n$, the identity
\begin{equation*}\label{singular-eigen-det}
 \prod_{1}^n |\lambda_k(T)| = \prod_{1}^n s_k(T) = |\det(T)| 
\end{equation*}
together with the inequality \ref{Weyl} for $r=n-1$, implies that 
\begin{equation}\label{eigen>sing}
    s_n(T)\leq |\lambda_n(T)|\, .
\end{equation}
\section{Main Results}

To apply the results of the previous section for bounding the derivative of the exponential map we need to work with a complex Lie algebra. For this reason we consider the complexification $\mathfrak{g}_{\mathds{C}}=\mathfrak{g}\otimes \mathds{C}$. The inner product of $\mathfrak{g}$ induced by the metric $d$, may be extended to an inner product in $\mathfrak{g}\otimes \mathds{C}$.
Moreover $(1-e^{-\textrm{ad}_x})/\textrm{ad}_x$ can be regarded as a linear operator on $\mathfrak{g}_{\mathds{C}}$.  Since the problem of estimating the norm the exponential map reduces to estimating the norm of a linear map over the Lie algebra, we can do estimates in $\mathfrak{g}_{\mathds{C}}$ and then derive the bounds for the real case. We will use the following notations:
\begin{notation}
Throughout this section $\mathfrak{g}_{\mathds{C}}=\mathfrak{g} \otimes \mathds{C}$ is the complexification of $\mathfrak{g}$ and we consider $\textrm{End}_{\mathds{C}}(\mathfrak{g}_{\mathds{C}})$ as a Banach space with the operator norm.
 \begin{enumerate}
    \item For any $a \in \mathds{C}$, $x \in \mathfrak{g}$, $T \in \textrm{End}_{\mathds{C}}(\mathfrak{g}_{\mathds{C}})$, $|a|$, $|x|$, and $\| T \|$ represent the absolute value of $a$, norm of vector $x$, and the operator norm of the linear map $T$, respectively. $\| T\| _{\mathfrak{g}}$ represents the norm of $T|_{\mathfrak{g}}$.
    \item For non-zero $x$, $\hat{x}=x/|x|$,  $\lambda_1,\cdots,\lambda_n,\  |\lambda_1|\geq \cdots \geq |\lambda_n|=0$ are the eigenvalues, and $s_1\geq \cdots \geq s_n$ are the singular values of $\textrm{ad}_{\hat{x}}$.
    $\tilde{\lambda}_1(x),\cdots, \tilde{\lambda}_n(x)$ are the corresponding eigenvalues of $(1-e^{-\textrm{ad}_x})/\textrm{ad}_x$. $\tilde{s}_1(x),\cdots, \tilde{s}_n(x)$ are singular values of $(1-e^{-\textrm{ad}_x})/\textrm{ad}_x$ in a non-increasing order. In addition, we admit the convention $(1-e^0)/0=1$.
    \item Let $\mathcal{B}$ be a given basis in $\mathfrak{g}$. For an endomorphism $T \in \textrm{End}_{\mathds{C}}(\mathfrak{g}_{\mathds{C}})$, $[T]_{\mathcal{B}}$ is the matrix representation of $T$ in basis $\mathcal{B}$. For a $n \times n$ square complex matrix $A$, $A_{\mathcal{B}} \in \textrm{End}_{\mathds{C}}(\mathfrak{g}_{\mathds{C}})$ is the unique endomorphism such that $[A_{\mathcal{B}}]_{\mathcal{B}}=A$. 
\end{enumerate}
\end{notation}
In order to prove Theorem 1 we need the following lemma: 
\begin{lem}\label{lem1}
Let $V$ be a finite dimensional complex inner product space, and $P,Q,T \in \textrm{End}_{\mathds{C}}(V)$, $P$ and $Q$ invertible. Then 
\begin{equation}
    \min_{|y|=1}|PTQ(y)| \geq \min_{|y|=1}|P(y)| \cdot \min_{|y|=1}|T(y)| \cdot \min_{|y|=1} |Q(y)|.
\end{equation}
\end{lem}
\begin{proof}
Without loss of generality we can assume $T$ is invertible (other wise the right side of inequality is zero). We have: 
\begin{equation*}
    \min_{|y|=1}|TQ(y)|= \min_{|y|=1}\left|T\left(\frac{Q(y)}{|Q(y)|}\right)|Q(y)|\right| \geq \min_{|y|=1}|T(y)| \cdot \min_{|y|=1} |Q(y)|\ , 
\end{equation*}
repeating the argument one more time gives the desired inequality.
\end{proof}
We may proceed to prove Theorem \ref{thm-main}:
\begin{proof}[Proof of Theorem \ref{thm-main}]

Assume $\textrm{ad}_x$ is diagonalizable, let $\mathcal{B}$ be an eigenbasis for $\textrm{ad}_{\hat{x}}$, and $\mathcal{F}$ an arbitray orthonormal basis. Let $P$ be the change of basis matrix from $\mathcal{F}$ to $\mathcal{B}$. Then $[\textrm{ad}_x]_{\mathcal{F}}=P^{-1}[\textrm{ad}_x]_{\mathcal{B}}P=P^{-1}\textrm{diag}(\lambda_{1}|x|,\cdots,\lambda_n|x|)P$ and thus
\begin{equation*}
    \left[\frac{1-e^{-\textrm{ad}_x}}{\textrm{ad}_x}\right]_{\mathcal{F}}=
    P^{-1} \textrm{diag}\left(\frac{1-e^{-\lambda_{1}|x|}}{\lambda_{1}|x|},\cdots,\frac{1-e^{-\lambda_{n}|x|}}{\lambda_{n}|x|}\right) P
\end{equation*}
So we have:  
\begin{equation*}
    \begin{split}
    & \max_{|y|=1,\, y\in \mathfrak{g}}\left| d\exp_{x} (y) \right| \leq 
    \left\|\frac{1-e^{-\textrm{ad}_x}}{\textrm{ad}_x}\right\|  \\
     & \leq  \|P_{\mathcal{F}}^{-1}\| \left\|  \textrm{diag}\left(\frac{1-e^{-\lambda_{1}|x|}}{\lambda_{1}|x|},\cdots,\frac{1-e^{-\lambda_{n}|x|}}{\lambda_{n}|x|}\right)_{\mathcal{F}}\right\|  \|P_{\mathcal{F}}\|
    \\
    & =\|P_{\mathcal{F}}^{-1}\|\|P_{\mathcal{F}}\| \tilde{\lambda}_{\max}(|x|)  \, ,
    \end{split}
\end{equation*}
In addition, using Lemma \ref{lem1} it follows that:
\begin{equation*}
    \begin{split}
    & \min_{|y|=1,\, y\in \mathfrak{g}}\left| d\exp_{x} (y) \right| \geq \min_{|y|=1,\, y\in \mathfrak{g}_{\mathds{C}}}
    \left|\frac{1-e^{-\textrm{ad}_x}}{\textrm{ad}_x}(y)\right|  \\
     & \geq  \min_{|y|=1} |P_{\mathcal{F}}^{-1}(y)| \cdot \min_{|y|=1} \left|\textrm{diag}\left(\frac{1-e^{-\lambda_{1}|x|}}{\lambda_{1}|x|},\cdots,\frac{1-e^{-\lambda_{n}|x|}}{\lambda_{n}|x|}\right)_{\mathcal{F}}(y)\right| \cdot
    \\
    & \, \quad \min_{|y|=1} |P_{\mathcal{F}}(y)|= \min_{|y|=1} |P_{\mathcal{F}}^{-1}(y)| \tilde{\lambda}_{\min}(|x|) \min_{|y|=1} |P_{\mathcal{F}}(y)|\, .
    \end{split}
\end{equation*}

\end{proof}

\begin{proof}[Proof of Theorem \ref{thm-MAIN}]

Let $\displaystyle \delta_0=\max \left\{ \left|\emph{ad}_{x_1}(y_1)\right|:\ x_1,y_1 \in\mathfrak{g}_{\mathds{C}},\  |x_1|,|y_1|=1 \right\}$ then $\|\textrm{ad}_x\| \leq \delta_0 |x|$. Noting  
\begin{equation}
    \frac{1-e^{-\textrm{ad}_x}}{\textrm{ad}_x}=\sum_{k=0}^{\infty} \frac{(-1)^k}{(k+1)!}\textrm{ad}^k_x
\end{equation}
and the absolute convergence of the above power series, we find that:
\begin{equation*}
    \begin{split}
        \left\| \frac{1-e^{-\textrm{ad}_x}}{\textrm{ad}_x} \right\|& = 
        \left\| \sum_{k=0}^{\infty} \frac{(-1)^k}{(k+1)!} \textrm{ad}_x^k
        \right\|  \leq  \sum_{k=0}^{\infty} \frac{1}{(k+1)!} \|\textrm{ad}_x^k\|
        \\
    & \leq \sum_{k=0}^{\infty} \frac{1}{(k+1)!} \|\textrm{ad}_x\|^k =   \frac{e^{\|\textrm{ad}_x\|}-1}{\|\textrm{ad}_x\|}  \leq \frac{e^{\delta_0|x|}-1}{\delta_0|x|}
    \end{split}
\end{equation*}
proving the upper bound. 

\par Applying the equation \ref{singular-eigen-det} we colnculde:
\begin{equation*}
    \begin{split}
      \left| d\exp_{x} (y) \right|  & \geq \tilde{s}_n(x) \geq \frac{\prod_1^n {|\tilde{\lambda}_k(x)|}}{(\tilde{s}_1(x))^{n-1}} = \frac{\prod_1^p {|\tilde{\lambda}_j(x)|}}{(\tilde{s}_1(x))^{n-1}} \\
       & \geq \left( \frac{e^{\delta_0|x|}-1}{\delta_0|x|} \right)^{1-n} \prod_1^p {|\tilde{\lambda}_j(x)|}\, .
    \end{split}
\end{equation*}
By the identity \ref{SpectralMT}:
\begin{equation*}
    \tilde{\lambda}_j(x)= \frac{1-e^{-\lambda_{j}|x|}}{\lambda_{j}|x|},\ 1\leq j \leq p \, ,
\end{equation*}
completing the lower bound proof.

The second statement in the theorem follows from the inequalities \ref{eigen<sing} and \ref{eigen>sing} applied to $T=(1-e^{-\textrm{ad}_x})/\textrm{ad}_x$.

\end{proof}
\begin{proof}[Proof of Proposition \ref{exp-nilpotent}]
Let $x \in \mathfrak{g}$ be non-zero, and Let $\displaystyle \delta_0$ be defined as in the statement of Theorem \ref{thm-MAIN}. We have:
\begin{equation*}
\begin{split}
    \|d\exp_x \| & = \left\| \frac{1-e^{-\textrm{ad}_x}}{\textrm{ad}_x}\right\|_{\mathfrak{g}} \leq \left\| \frac{1-e^{-\textrm{ad}_x}}{\textrm{ad}_x}\right\| = \left\| \sum_{k=0}^{p-1} \frac{(-1)^k}{(k+1)!} |x|^k \textrm{ad}_{\hat{x}}^k \right\|  \\
     & \leq \sum_{k=0}^{p-1} \frac{1}{(k+1)!} |x|^k \|\textrm{ad}_{\hat{x}}\|^k \leq \sum_{k=0}^{p-1} \frac{1}{(k+1)!} |x|^k \delta_0^k =O(|x|^{p-1})\, .
\end{split}
\end{equation*}
Let 
\begin{equation*}
    Q_1(|x|)=\sum_{k=0}^{p-1} \frac{1}{(k+1)!} |x|^k \delta_0^k, \quad Q=Q_1^{n-1} 
\end{equation*}
then $\tilde{s}_1(x) \leq Q_1(|x|)$. Noting $\prod_1^n {|\tilde{\lambda}_k(x)|}=1$ we have:
\begin{equation*}
    \min_{|y|=1} \left| d\exp_{x} (y) \right| \geq \tilde{s}_n(x) \geq \frac{\prod_1^n {|\tilde{\lambda}_k(x)|}}{(\tilde{s}_1(x))^{n-1}} \geq \frac{1}{Q(|x|)}\, . 
\end{equation*}
\end{proof}

\bibliography{aomsample}
\bibliographystyle{aomalpha}

\end{document}